  \documentclass[12pt,reqno ]{amsart}

\usepackage{amssymb}

\usepackage{mathrsfs}
\usepackage{xcolor}
\usepackage{graphicx}
\usepackage[shortlabels]{enumitem}
\usepackage{mathtools}
\usepackage{amsthm}
\usepackage{thmtools}
\usepackage{tikz}
\usepackage{scalerel,stackengine}
\stackMath
\newcommand\reallywidehat[1]{%
\savestack{\tmpbox}{\stretchto{%
  \scaleto{%
    \scalerel*[\widthof{\ensuremath{#1}}]{\kern-.6pt\bigwedge\kern-.6pt}%
    {\rule[-\textheight/2]{1ex}{\textheight}}
  }{\textheight}%
}{0.5ex}}%
\stackon[1pt]{#1}{\tmpbox}%
}
\allowdisplaybreaks

\def\Xint#1{\mathchoice
{\XXint\displaystyle\textstyle{#1}}%
{\XXint\textstyle\scriptstyle{#1}}%
{\XXint\scriptstyle\scriptscriptstyle{#1}}%
{\XXint\scriptscriptstyle\scriptscriptstyle{#1}}%
\!\int}
\def\XXint#1#2#3{{\setbox0=\hbox{$#1{#2#3}{\int}$ }
\vcenter{\hbox{$#2#3$ }}\kern-.6\wd0}}

\def\fint{\Xint-}

\usetikzlibrary{cd}
\usepackage[unicode,bookmarksopen,bookmarksdepth=3]{hyperref}
\usepackage{xcolor}
\hypersetup{
    colorlinks,
    linkcolor={red!80!black},
    citecolor={blue!60!black},
    urlcolor={blue!80!black}
}

\newtheorem{thm}{Theorem}[section]
\newtheorem{prop}[thm]{Proposition}
\declaretheoremstyle[notefont=\bfseries,notebraces={}{},%
    headpunct={},postheadspace=1em]{mystyle}

\newtheorem{lem}[thm]{Lemma}

\theoremstyle{definition}

\theoremstyle{remark}

\numberwithin{equation}{section}

\newcommand{\BMO}{\text{BMO}}

\newcommand{\reals}{\mathbb{R}}
\newcommand{\cplx}{\mathbb{C}}

\DeclareMathOperator{\dee}{d}
\newcommand{\dx}{\dee \! x}
\newcommand{\dy}{\dee \! y}

\newcommand{\dt}{\dee \! t}

\newcommand{\al}{\alpha}
\newcommand{\be}{\beta}

\newcommand{\innp}[1]{\left \langle #1 \right \rangle}

\newcommand{\MS}[1]{\mathscr{#1}}

 \newcommand{\norm}[1]{\ensuremath{\left\|#1\right\|}}  \newcommand{\abs}[1]{\ensuremath{\left\vert#1\right\vert}}  \newcommand{\pr}[1]{\ensuremath{\left(#1\right)}}

\newcommand{\MC}[1]{\ensuremath{\mathcal{#1}}}   \newcommand{\D}{\mathcal{D}}   \newcommand{\T}[1]{\text{#1}} \newcommand{\R}{\mathbb{R}}   \newcommand{\N}{\mathbb{N}}     \newcommand{\V}[1]{\vec{#1}}  \newcommand{\Cn}{\ensuremath{\mathbb{C}^n}}

 \newcommand{\f}[2]{\frac{#1}{#2}}  

  \newcommand{\Rd}{\mathbb{R}^d}

         \renewcommand{\v}[1]{\ensuremath{\vec{#1}}} \newcommand{\br}[1]{{\ensuremath{\left[ #1\right]}}}

\newcommand{\BMOVUpq}{\ensuremath{{{\text{BMO}}_{V, U} ^{p, q}}}}
 
  \newcommand{\BMOVUTpq}{\ensuremath{{{\widetilde{\text{BMO}}}_{V, U} ^{p, q}}}}
  \newcommand{\BMOVUTpqd}{\ensuremath{{{\widetilde{\text{BMO}}}_{U', V'} ^{q', p'}}}}

  \newcommand{\Mn}{\ensuremath{\mathbb{M}_{n \times n}}}

     \newcommand{\inrd}{\ensuremath{\int_{\mathbb{R}^d}}}  
   
\newcommand{\Apq}[1]{\ensuremath{\norm{#1}_{\text{A}_{p, q}}}}
\newcommand{\Aq}[1]{\ensuremath{\norm{#1}_{\text{A}_{ q}}}}

\newcommand{\mymor}[3]{#1\colon#2\longrightarrow#3}

\newcommand{\mynorm}[1]{\left| \left | #1 \right| \right|}				 	 

\newcommand{\myset}[1]{ \left\{#1\right\} }

\newcommand{\myinnerprod}[2]{\left< #1, #2 \right>}

\newcommand{\myabs}[1]{\left| #1 \right|}
\newcommand{\set}[1]{\left\{#1\right\}}
\newcommand{\ol}{\overline}

\tikzcdset{row sep/tinyrow/.initial=-0.5em}

\begin{document}

\title[Two Weighted Inequalities for Commutators]{Two Matrix Weighted Inequalities for Commutators with Fractional Integral Operators}


\author{Roy Cardenas}
\address{University at Albany, SUNY, Department of Mathematics, 1400 Washington Ave, Albany, NY 12222}
\curraddr{}
\email{rcardenas@albany.edu}
\thanks{}

\author{Joshua Isralowitz}
\address{University at Albany, SUNY, Department of Mathematics, 1400 Washington Ave, Albany, NY 12222}
\curraddr{}
\email{jisralowitz@albany.edu}
\thanks{}


\date{}

\dedicatory{}

\commby{}

\begin{abstract}
In this paper we prove two matrix weighted norm inequalities for the commutator of a fractional integral operator and multiplication by a matrix symbol.  More precisely, we extend the recent results of the second author, Pott, and Treil on two matrix weighted norm inequalities for commutators of Calderon-Zygmund operators and multiplication by a matrix symbol to the fractional integral operator setting.  In particular, we completely extend the fractional Bloom theory of Holmes, Rahm, and Spencer to the two matrix weighted setting with a matrix symbol.
\end{abstract}

\maketitle


\section{Introduction}

Let $w$ be a weight on $\Rd$ and let $L^p(w)$ be the standard weighted Lebesgue space with respect to the norm \begin{equation*} \|f\|_{L^p( w)} = \left(\inrd |f(x)|^p w(x) \, dx \right)^\frac{1}{p}. \end{equation*} Furthermore, let A${}_{p, q}$ for $p, q > 1$ be the  Muckenhoupt class of weights $w$ satisfying \begin{equation*} \sup_{\substack{Q \subseteq \Rd \\ Q \text{ is a cube}}} \left( \fint_Q w (x) \, dx \right)\left( \fint_Q w ^{-\frac{p'}{q}}     (x) \, dx\right)^{\frac{q}{p'}} < \infty \end{equation*} where $\fint_Q$ is the unweighted average over $Q$ (which will also occasionally be denoted by $m_Q$). When $p = q$ we write $A_{p} := A_{p, p}$ as usual.

Given a weight $\nu$, we say $b \in \text{BMO}_{\nu}$ if  \begin{equation*} \|b\|_{\text{BMO}_\nu} = \sup_{\substack{Q \subseteq \Rd \\ Q \text{ is a cube}} } \frac{1}{\nu(Q)} \int_Q |b(x) - m_Q b | \, dx  < \infty \end{equation*} (where $\nu(Q) = \int_Q \nu$) so that clearly
 $\text{BMO} = \text{BMO}_\nu$ when $\nu \equiv 1$.  Further, given a linear operator $T,$ define the commutator $[M_b, T] = M_b T - TM_b$ with $M_b$ being multiplication by $b$.  In the papers \cite{HLW1,HLW2} the authors extended earlier work of S. Bloom \cite{BL} and proved that if $u, v \in \text{A}_p$ and $T$ is any Calder\'{o}n-Zygmund operator (CZO) then \begin{equation} \|[M_b, T]\|_{L^p(u) \rightarrow L^p(v)} \lesssim \|b\|_{\text{BMO}_\nu} \label{HWUpper}  \end{equation} where $\nu = (uv^{-1})^\frac{1}{p}$ and it was proved in \cite{HLW2} that if $R_s$ is the $s^\text{th}$ Riesz transform then \begin{equation} \|b\|_{\text{BMO}_\nu} \lesssim \max_{1 \leq s \leq d } \|[ M_b, R_s]\|_{L^p(u) \rightarrow L^p(v )}.  \label{HWLower}  \end{equation}

Furthermore, let $I_\alpha$ be the fractional integral operator defined by the formula
\begin{equation*}
    I_\alpha {f}(x) = \int_{\reals^d} \frac{{f}(y)}{\myabs{x-y}^{d - \alpha}} \dy, \text{ for } 0 < \alpha < d.
\end{equation*}

\noindent It was proved in \cite{HRS} that if $0 < \alpha < d$ and  $\alpha/d + 1/q = 1/p$,  if $u, v \in A_{p, q}$, and if $\nu = u^\frac{1}{q} v^{-\frac{1}{q}}$ then \begin{equation} \|[M_b, I_\al]\|_{L^p(u^\frac{p}{q}) \rightarrow L^p(v)} \approx \|b\|_{\BMO(\nu)} \label{HRSBloom}  \end{equation}

On the other hand, matrix weighted extensions and generalizations of \eqref{HWUpper} and \eqref{HWLower} that surprisingly hold for two arbitrary matrix weights (and provided new results even in the scalar $p = 2$ setting of a single scalar weight) were proved in \cite{IPT}, and it is the purpose of this paper to extend the results of \cite{IPT} to the fractional setting,  providing matrix weighted extensions of \eqref{HRSBloom} that hold for two arbitrary matrix weights. Note that for the rest of this paper we will assume that $0 < \alpha < d$ and $\alpha/d + 1/q = 1/p$.

  In particular, for any linear operator $T$ acting on scalar valued functions on $\Rd$, we can canonically extend $T$ to act on $\Cn$ valued functions $\V{f}$ by the formula $T\V{f} := \sum_{j = 1}^n \pr{T \innp{\V{f}, \V{e}_j}_{\Cn} } \V{e}_j$ where $\{\V{e}_j\}$ is any orthonormal basis of $\Cn$ (and note that this is easily seen to be independent of the orthonormal basis chosen.)  Let $W : \Rd \rightarrow \Mn$ be  an $n \times n$ matrix weight  (a positive definite a.e. $\Mn$ valued function on $\Rd$) and let $L^p(W)$ be the space of $\Cn$ valued functions $\v{f}$ such that \begin{equation*} \|\v{f}\|_{L^p(W)}  = \left(\inrd |W^\frac{1}{p}(x) \v{f}(x)|^p \, dx \right)^\frac{1}{p} < \infty. \end{equation*}   Furthermore, for $p, q > 1$ we will say that a matrix weight $W$ is a matrix A${}_{p, q}$ weight (see \cite{IM}) if it satisfies \begin{equation} \label{MatrixApqDef} \Apq{W} = \sup_{\substack{Q \subset \R^d \\ Q \text{ is a cube}}} \fint_Q \left( \fint_Q \|W^{\frac{1}{q}} (x) W^{- \frac{1}{q}} (y) \|^{p'} \, dy \right)^\frac{q}{p'} \, dx  < \infty \end{equation} and when $p = q$ we say $W$ is a matrix A${}_p$ weight (see \cite{R}).

Now for scalar weights $u$ and $v$,  notice that by multiple uses of the A${}_q$ property and H\"{o}lder's inequality we have \begin{align*} m_Q \nu   \approx (m_Q u) ^\frac{1}{q} (m_Q v^{-\frac{q'}{q}})^\frac{1}{q'}   \approx (m_Q u) ^\frac{1}{q} (m_Q v)^{-\frac{1}{q}}  \approx (m_Q u^\frac{1}{q}) (m_Q v^\frac{1}{q})^{-1}.  \end{align*}  Thus, $b \in \text{BMO}_\nu$ when $u$ and $v$ are A${}_q$ weights if and only if

\begin{equation*} \sup_{\substack{Q \subseteq \R \\ Q \text{ is a cube}} }  \fint_Q (m_Q v^\frac{1}{q}) (m_Q u^\frac{1}{q})^{-1}  |b(x) - m_Q b | \, dx  < \infty, \end{equation*} which is a condition that easily extends to the matrix weighted setting, noting that  $A_{p, q} \subset A_q$ when $q > p$, since then $q' < p'$ and so H\"{o}lder's inequality gives us \begin{equation*} \fint_Q \left( \fint_Q \|W^{\frac{1}{q}} (x) W^{- \frac{1}{q}} (y) \|^{q'} \, dy \right)^\frac{q}{q'} \, dx \leq \fint_Q \left( \fint_Q \|W^{\frac{1}{q}} (x) W^{- \frac{1}{q}} (y) \|^{p'} \, dy \right)^\frac{q}{p'} \, dx. \end{equation*}

  Namely, if $U, V$ are $n \times n$ matrix A${}_{p, q}$ weights, then we define $\BMOVUpq$ to be the space of $n \times n$ locally integrable matrix functions $B$  where \begin{equation*} \|B\|_{\BMOVUpq} = \sup_{\substack{Q \subseteq \Rd \\ Q \text{ is a cube}} }  \pr{\fint_Q \|(m_Q V^\frac{1}{q}) (B(x) - m_Q B)(m_Q U^\frac{1}{q})^{-1} \| \, dx}^\f{1}{q}  < \infty \end{equation*} so that $\|b\|_{\BMOVUpq}  \approx \|b\|_{\text{BMO}{\nu}}$  if $U, V$ are scalar weights and $b$ is a scalar function. Note that the $\BMOVUpq$ condition is much more naturally defined in terms of reducing matrices, which will be discussed in Section \ref{JNFracSec}.

   We will need a definition before we state our first result.  We say that a linear operator $R$ acting on  scalar functions is a fractional lower bound operator if for any $n \in \N$ and any $n \times n$  matrix weight $W$ we have  \begin{equation} \Apq{W} ^\frac{1}{q}  \lesssim \|T\|_{L^p(W^\frac{p}{q}) \rightarrow L^q(W)} \label{LBO} \end{equation} \noindent with the bound independent of $W$ (but not necessarily independent of $n$), and $\|T\|_{L^p(W^\frac{p}{q}) \rightarrow L^q(W)} < \infty$ if $W$ is a matrix A${}_{p, q}$ weight.

\begin{thm} \label{BloomFrac}  Let $T$ be any linear operator acting on scalar valued functions where its canonical $\Cn$ valued extension is bounded from $L^p(W^\frac{p}{q})$ to $L^q(W)$ for all $n \times n$ matrix A${}_{p, q}$ weights $W$ and all $n \in \N$ with bound depending on $T, n, d, p$, and $\Apq{W}$ (which is known to be true for fractional integral operators, see \cite[Theorem 1.4]{IM}.)  If $U, V$ are $m \times m$ matrix A${}_p$ weights and $B$ is an $m \times m$ locally integrable matrix function for some $m \in \N$,  then \begin{equation} \|[M_B, T]\|_{L^p(U^{\frac{p}{q}}) \rightarrow L^q(V)} \lesssim \|B\|_{\BMOVUpq} \label{RoyJoshBloomUpper}\end{equation} with bounds depending on $T, m, d, p, \Apq{U}$ and $\Apq{V}$.

Furthermore, for any fractional lower bound operator $T$ we have the lower bound estimate \begin{equation} \|B\|_{\BMOVUpq} \lesssim \|[M_B, T]\|_{L^p(U^{\frac{p}{q}}) \rightarrow L^q(V)} \label{RoyJoshBloomLower} \end{equation}
\end{thm}

 Like in \cite{IPT},  we will use matrix weighted arguments inspired by \cite{GPTV} in the next section to prove Theorem \ref{BloomFrac} in terms of a weighted BMO quantity $\|B\|_{\BMOVUTpq}$ that is equivalent to $\|B\|_{\BMOVUpq}$ when $U$ and $V$ are matrix A${}_{p, q}$ weights (see Theorem \ref{JNFrac}) but is much more natural for more arbitrary matrix weights $U$ and $V$.  More precisely, define  \begin{equation} \|B\|_{\BMOVUTpq}^q  = \sup_{\substack{Q \subseteq \Rd \\ Q \text{ is a cube}} } \fint_Q \pr{\fint_Q \norm{V^\frac{1}{q} (x) (B(x) - B(y)) U^{-\frac{1}{q}}(y) }^{p'} \, dy }^\frac{q}{p'} \, dx. \label{BMOT} \end{equation}
 \noindent

 We will then give relatively short proofs of the following two results in Section \ref{IntSec}.

\begin{lem} \label{IntUBFrac} Let $T$ be any linear operator defined on scalar valued functions where its canonical $\Cn$ valued extension $T$ for any $n \in \N$ satisfies \begin{equation*} \|T \|_{L^p(W^\frac{p}{q}) \rightarrow L^q(W)} \leq  \phi (\Apq{W}) \end{equation*} for some positive increasing function $\phi$ (possibly depending on $T, d, n, p, q$.)   If $U, V$ are $m \times m$ matrix A${}_{p, q}$ weights and $B$ is a locally integrable $m \times m$  matrix valued function for some $m \in \N$, then   \begin{equation*} \|[M_B, T ]\|_{L^p(U^{\frac{p}{q}}) \rightarrow L^q(V)} \leq   \|B\|_{\BMOVUTpq} \phi\pr{3^{\frac{q}{p' }} \pr{\Apq{U} + \Apq{V} } + 1}  \end{equation*}\end{lem}

\begin{lem} \label{IntLBFrac}  If $T$ is any fractional lower bound operator then for any $m \times m$ matrix A${}_{p, q}$ weights $U, V$ and an $m \times m$ matrix symbol $B$ we have \begin{equation*} \|B\|_{\BMOVUTpq}  \lesssim   \|[M_B, T]\|_{L^p(U^{\frac{p}{q}}) \rightarrow L^q(V)}\end{equation*}  where the bound depends possibly on $n, p, d$ and $T$ but is independent of $U$ and $V$.  \end{lem}


 As in \cite{IPT}, we will prove that the fractional integral operator is a fractional lower bound operator in Section \ref{WienerSection} by utilizing the Schur multiplier/Wiener algebra ideas from \cite{LT}, and thus recover \eqref{RoyJoshBloomLower}.  These arguments will in fact prove the following (see \cite{IPT} for an analogous result with respect to the Riesz transforms).  Here, for ease of notation, we set $U' = U^{-\frac{p'}{q}}$ and $V' = V^{-\frac{p'}{q}}$.

\begin{thm} \label{FracLowerBoundThm} Let $U$ and $V$ be any (not necessarily A${}_p$) matrix weights.  If $B$ is any locally integrable $m \times m$  matrix valued function then \begin{equation} \max\left\{\|B\|_{\BMOVUTpq}, \|B\|_{\BMOVUTpqd}\right\}\lesssim   \|[M_B, I_\al]\|_{L^p(U^\frac{p}{q} ) \rightarrow L^q(V)}. \label{RieszThmA}\end{equation} \noindent
\end{thm}

\noindent Note that the two quantities $\|B\|_{\BMOVUTpq}$ and $ \|B\|_{\BMOVUTpqd}$ are equivalent when $U, V \in \T{A}_{p, q}$ (which will be proved in Section \ref{WienerSection}) and in general should be thought of as ``dual" matrix weighted BMO quantities.   Finally, we will show that an Orlicz ``bumped" version of these conditions are sufficient for the general two matrix weighted boundedness of fractional integral operators. In particular, we will prove the following result in Section \ref{OrliczUpBound} (see \cite{IPT} for an analogous result for Calderon-Zygmund operators)

\begin{prop} \label{FracOrliczProp} Let  $U$ and $V$ be any $m \times m$ matrix weights, and suppose that $C$ and $D$ are Young
   functions with $\bar{D}\in B_{p,q}$ and $\bar{C}\in B_{q'}$.

 Then \begin{equation*} \|[M_B, I_{\al}]\|_{L^p(U^\frac{p}{q}) \rightarrow L^q(V)} \lesssim \min\{\kappa_1, \kappa_2\} \end{equation*}  where \[ \begin{split}\kappa_{1} & = \sup_{Q}\|\|V^\frac{1}{q} (x) (B(x)-B(y)) U^{-\frac{1}{q}} (y) \|_{C_{x},Q}\|_{D_{y},Q}\\ \kappa_{2} & =\sup_{Q}\|\|V^\frac{1}{q} (x) (B(x)-B(y)) U^{-\frac{1}{q}} (y) \|_{D_{y},Q}\|_{C_{x},Q} \end{split} \]

\end{prop}

\noindent We refer the reader to Section $5.2$ in \cite{CIM} for the standard Orlicz space related definitions used in the statement of Proposition \ref{FracOrliczProp}.

 It is important to emphasize that Theorem \ref{FracLowerBoundThm} and Proposition \ref{FracOrliczProp} are new,  even in the scalar setting of a single weight.

\section{Intermediate fractional upper and lower bounds} \label{IntSec}

As stated in the introduction, we will give short proofs of Lemma \ref{IntUBFrac} and Lemma \ref{IntLBFrac} in this section, beginning with Lemma \ref{IntUBFrac}.

\subsection{Proof of Lemma \ref{IntUBFrac}} \label{IntUBFracSec}

Define the $2 \times 2$ block matrix-valued function $\Phi = \mymor{\Phi_{U,V,B}}{\reals^d}{M_{2m}(\cplx)}$ by

\[
\Phi(x):=
\begin{pmatrix}
V^{\frac{1}{q}}(x) & 0 \\

0 & U^{\frac{1}{q}}(x)
\end{pmatrix}
\begin{pmatrix}
I & B(x) \\
0 & I
\end{pmatrix}
=
\begin{pmatrix}
V^{\frac{1}{q}}(x) & V^{\frac{1}{q}}(x)B(x) \\
0 & U^{\frac{1}{q}}(x)
\end{pmatrix},
\]
so that for a.e. $x \in \reals^d$,
\[
\Phi^{-1}(y)=
\begin{pmatrix}
V^{-\frac{1}{q}}(y) & -B(y)U^{-\frac{1}{q}}(y)  \\
0 & U^{-\frac{1}{q}}(y)
\end{pmatrix}.
\]
 Thus, we have
\[
\Phi T \Phi^{-1} =
\begin{pmatrix}
V^{\frac{1}{q}}TV^{-\frac{1}{q}} & V^{\frac{1}{q}}[M_B,T]U^{-\frac{1}{q}} \\
0 &U^{\frac{1}{q}}TU^{-\frac{1}{q}}
\end{pmatrix}.
\]
 Note that $W := \left(\Phi^*\Phi\right)^{\frac{q}{2}}$ is a matrix weight and, by polar decomposition, there exists a unitary a.e. matrix function $\mathscr{U}$ such that $\Phi(x) = \mathscr{U}(x)W^\frac{1}{q}(x)$. This gives us that
\begin{align}
    &\mynorm{T}_{L^p\left( W^\frac{p}{q} \right) \longrightarrow L^q(W)}
    \phantom{A}= \mynorm{W^\frac{1}{q}TW^{-\frac{1}{q}}}_{L^p \rightarrow L^q} \nonumber \\
    &\phantom{A}= \mynorm{\Phi T \Phi^{-1}}_{L^p \rightarrow L^q}
    \phantom{A}\approx \max \left\{ \mynorm{V^{\frac{1}{q}}TV^{-\frac{1}{q}}}_{L^p \rightarrow L^q}, \mynorm{V^{\frac{1}{q}}[M_B,T]U^{-\frac{1}{q}}}_{L^p \rightarrow L^q}, \mynorm{U^{\frac{1}{q}}TU^{-\frac{1}{q}}}_{L^p \rightarrow L^q} \right\}  \nonumber \\
    &\phantom{A}\geq \mynorm{V^{\frac{1}{q}}[M_B,T]U^{-\frac{1}{q}}}_{L^p \rightarrow L^q}
    \phantom{A}= \mynorm{[M_B, T]}_{L^p\left( U^\frac{p}{q} \right) \longrightarrow L^q(V)} \nonumber
\end{align}
Using the assumption in Lemma \ref{IntUBFrac} that $\mynorm{T}_{L^p\left( W^\frac{p}{q} \right) \longrightarrow L^q(W)} \lesssim \phi(\Apq{W})$ we get that
\begin{equation} \label{eq-comm-est}
    \mynorm{[M_B, T]}_{L^p\left( U^\frac{p}{q} \right) \longrightarrow L^q(V)} \lesssim \phi(\Apq{W}).
\end{equation}
Unravelling the $A_{p,q}$ condition for $W$, we obtain
\begin{align*}
    &\Apq{W} =\sup_Q \fint_Q \left( \fint_Q \mynorm{W^\frac{1}{q}(x) W^{-\frac{1}{q}}(y)}^{p'} \dy \right)^\frac{q}{p'} \dx \\
    &\phantom{A}=\sup_Q \fint_Q \left( \fint_Q \mynorm{\Phi(x) \Phi^{-1}(y)}^{p'} \dy \right)^\frac{q}{p'} \dx \\
    &\phantom{A}\leq 3^{\frac{q}{p'}} \pr{ \mynorm{U}_{A_{p,q}} + \mynorm{V}_{A_{p,q}} + \sup_Q \fint_Q \left( \fint_Q \mynorm{V^{\frac{1}{q}}(x) \left( B(x) - B(y)\right) U^{-\frac{1}{q}}(y)}^{p'} \dy \right)^\frac{q}{p'} \dx} \\
    &\phantom{A}  =  3^{\frac{q}{p'}} \pr{ \mynorm{U}_{A_{p,q}} + \mynorm{V}_{A_{p,q}} + \mynorm{B}^q_{\widetilde{\T{BMO}}^{p,q}_{V,U}}}
\end{align*}
and thus
\begin{equation*}
    \mynorm{[M_B, I_\alpha]}_{L^p\left( U^\frac{p}{q} \right) \longrightarrow L^q(V)} \lesssim \phi \pr{ 3^{\frac{q}{p'}}  \left( \mynorm{U}_{A_{p,q}} + \mynorm{V}_{A_{p,q}} + \mynorm{B}^q_{\widetilde{BMO}^{p,q}_{V,U}} \right)}.
\end{equation*}
Re-scaling with $B$ replaced by $B\mynorm{B}^{-1}_{\widetilde{BMO}^{p,q}_{V,U}}$ now completes the proof.


\subsection{Proof of Lemma \ref{IntLBFrac}} We now prove Lemma \ref{IntLBFrac}.  Let $W$ and $\Phi$ be defined as in the previous subsection, so that
  \begin{align*}  &\left(\Apq{U} + \Apq{V} + \|B\|_{\BMOVUTpq} ^q  \right)^\frac{1}{q}
   \approx \Apq{W} ^\frac{1}{q} \lesssim \mynorm{T}_{L^p\left( W^{\frac{p}{q}}\right) \longrightarrow L^q(W)} \nonumber \\
    &\phantom{W}\lesssim \mynorm{\left[M_B, T \right]}_{L^p\left( U^{\frac{p}{q}}\right) \longrightarrow L^q(V)} + \mynorm{T }_{L^p\left( U^{\frac{p}{q}}\right) \longrightarrow L^q(U)}
        + \mynorm{T}_{L^p\left( V^{\frac{p}{q}}\right) \longrightarrow L^q(V)}. \nonumber
\end{align*}
Clearly we may assume that $\mynorm{\left[M_B, T \right]}_{L^p\left( U^{\frac{p}{q}}\right) \longrightarrow L^q(V)} < \infty$ and so by assumption all quantities above are finite.  Re-scalling $B \mapsto rB$ for $r>0$, dividing by $r$, and taking $r \longrightarrow \infty$, we obtain
\begin{equation*}
    \mynorm{B}_{\widetilde{BMO}^{p,q}_{V,U}} \lesssim \mynorm{\left[M_B,I_\alpha \right]}_{L^p\left( U^{\frac{p}{q}}\right) \longrightarrow L^q(V)}.
\end{equation*}
which is the desired lower bound.

\section{Proof of Theorem \ref{BloomFrac}} \label{JNFracSec} We now prove Theorem \ref{BloomFrac} (assuming that $I_\al$ is a fractional lower bound operator, which will be proved in the next section) by proving that $\|B\|_{\BMOVUpq} \approx \|B\|_{\BMOVUTpq}$  when $U, V$ are matrix A${}_{p, q}$ weights (see Theorem \ref{JNFrac}).  To do this we need the concept of a reducing matrix.  In particular,  for any norm $\rho$ on $\Cn$ there exists a positive definite $n \times n$ matrix $A$ where for any $\V{e} \in \Cn$ we have $$n^{-1} \abs{A\V{e}} \leq \rho(\V{e}) \leq \abs{A\V{e}}$$ (see \cite[Lemma 11.4]{NT}).

In particular, for any matrix weight $U$ and measurable $0 < |E| < \infty $ there exists $n \times n$ matrices $\MC{U}_E, \MC{U}_E '$ where for any $\V{e} \in \Cn$ we have $$\abs{\MC{U}_E \V{e}} \approx \pr{ \fint_E \abs{U^\frac{1}{q} (x) \V{e}}^{q} \, dx}^\frac{1}{q}, \qquad \abs{\MC{U}_E '  \V{e}} \approx \pr{ \fint_E \abs{U^{-\frac{1}{q}} (x) \V{e}}^{p'} \, dx}^\frac{1}{p'}.$$  Similarly for a matrix weight $V$ we will use the notation $\MC{V}_E$ and $\MC{V}_E '$ for these reducing matrices.  Using reducing matrices in conjunction with elementary linear algebra, it is easy to see that for a matrix weight $U$ we have  $$\Apq{U} ^\frac{1}{q} \approx  \sup_Q \norm{\MC{U}_Q \MC{U}_Q ' } = \sup_Q \norm{\MC{U}_Q '  \MC{U}_Q  } \approx \sup_{Q} \pr{\fint_Q \left( \fint_Q \|W^{\frac{1}{q}} (x) W^{- \frac{1}{q}} (y) \|^{q} \, dx \right)^\frac{p'}{q } \, dy}^\frac{1}{p'}$$ and similarly an easy application of H\"{o}lder's inequality gives us that $$\abs{\innp{\V{e}, \V{f}}_{\Cn}} \leq \abs{\MC{U}_Q \V{e}} \abs{\MC{U}_Q ' \V{f}}$$ for  $\V{e}, \V{f} \in \Cn$, which clearly implies that \begin{equation}  \norm{\MC{U}_Q ^{-1} (\MC{U}_Q ')^{-1}} \leq 1 \label{MatrixHolder}. \end{equation}

 We now prove the following matrix weighted John-Nirenberg type theorem, which should be thought of as a fractional generalization of the matrix weighted John-Nirenberg theorem from \cite{IPT}.

\begin{thm} \label{JNFrac}  If $U, V$ are two $m \times m$ matrix weights such that $U, V \in \T{A}_{p, q}$ and $B$ is an $m \times m$ locally integrable matrix function, then the following are equivalent (where the suprema is taken over all cubes $Q$).
\begin{itemize}
\item[(1)] $\displaystyle \sup_Q \fint_Q \norm{{\MC{V}}_Q (B(x) - m_Q B) {\MC{U}}_Q ^{-1}} \, dx $
\item[(2)] $\displaystyle \sup_Q \pr{\fint_Q \norm{V^\frac{1}{q} (x)  (B(x) - m_Q B) {\MC{U}}_Q ^{-1}} ^q  \, dx}^\frac{1}{q} $
\item[(3)] $\displaystyle \sup_Q \pr{\fint \norm{U^{-\frac{1}{q}} (x)  (B^*(x) - m_Q B^*) ({\MC{V}} ' _Q) ^{-1}} ^{p'}  \, dx}^\frac{1}{p'} $
\item[(4)] $\displaystyle \sup_Q \pr{\fint_Q \pr{ \fint_Q \norm{V^\frac{1}{q} (x)  (B(x) - B(y)) U^{-\frac{1}{q}} (y) } ^{p'}  \, dy}^\frac{q}{p'} \, dx}^\frac{1}{q} $
\item[(5)] $\displaystyle \sup_Q \pr{\fint_Q \pr{ \fint_Q \norm{V^\frac{1}{q} (x)  (B(x) - B(y)) U^{-\frac{1}{q}} (y) } ^{q}  \, dx}^\frac{p'}{q} \, dy}^\frac{1}{p'} $
  \item[(6)] $\displaystyle \sup_Q \pr{\fint_Q \pr{ \fint_Q \norm{V^\frac{1}{q} (x)  (B(x) - B(y)) U^{-\frac{1}{q}} (y)} ^{q}  \, dx}^\frac{q'}{q} \, dy}^\frac{1}{q'} $
\end{itemize}
\end{thm}

\noindent Note that Lemma $2.2$ in \cite{IKP} says that $|\MC{U}_Q \V{e}| \approx |m_Q (U^\frac{1}{q}) \V{e}|$ if $U \in \T{A}_q$   so that $(1) \approx \|B\|_{\BMOVUpq}$ for $U, V \in \T{A}_{p, q}$.

Before we prove Theorem \ref{JNFrac} we need to discuss some duality properties of matrix A${}_{p, q}$ weights.  To better keep track of the exponents and matrix weights that corresponding to a reducing matrix we temporarily use the notation $\MS{V}_Q (W, q), \MS{V}_Q ' (W, p, q) $ to denote reducing matrices where  $$\abs{\MS{V}_Q (W, q) \V{e}} \approx \pr{ \fint_Q \abs{W^\frac{1}{q} (x) \V{e}}^{q} \, dx}^\frac{1}{q}, \qquad \abs{\MS{V}_Q ' (W, p, q)  \V{e}} \approx \pr{ \fint_Q \abs{W^{-\frac{1}{q}} (x) \V{e}}^{p'} \, dx}^\frac{1}{p'}$$ so that $$\Apq{W} \approx \sup_Q \norm{\MS{V}_Q(W, q) \MS{V}_Q '(W, p, q) }^{p'}, \qquad \Aq{W} \approx \sup_Q \norm{\MS{V}_Q(W, q) \MS{V}_Q'  (W, q, q) }^q.$$  Moreover \begin{equation}\abs{\MS{V}_Q (W^{-\frac{p'}{q}}, p' ) \V{e}} \approx \pr{ \fint_Q \abs{W^{-\frac{1}{q}} (x)  \V{e}}^{p'} \, dx}^\frac{1}{p'} \approx \abs{\MS{V}_Q ' (W, p, q)  \V{e}}\label{DualRedEq1}\end{equation} and similarly \begin{equation} \abs{\MS{V}_Q ' (W^{-\frac{p'}{q}}, q', p' ) \V{e}} \approx \pr{ \fint_Q \abs{W^{\frac{1}{q}} (x) \V{e}}^{q} \, dx}^\frac{1}{q} \approx \abs{\MS{V}_Q (W,  q)  \V{e}}, \label{DualRedEq2}\end{equation} which (as observed in \cite{IM}) means that $W \in \T{A}_{p, q}$ if and only if $W^{-\frac{p'}{q}} \in \T{A}_{q', p'}$.   Note that with this notation we have  $$\MC{V}_Q = \MS{V}_Q (V, q), \qquad \MC{V}_Q ' = \MS{V}_Q ' (V, p, q)$$ and a similar statement holds for $U$.

\begin{proof}[Proof of Theorem \ref{JNFrac}]

Recall from the introduction that $\T{A}_{p, q} \subseteq \T{A}_q$.  Thus, we have from \cite[Corollary 4.7]{IPT} that $(1) \Longleftrightarrow (2) \Longleftrightarrow(6)$.  Moreover, since $U, V \in \T{A}_{p, q}$ if and only if $U^{-\frac{p'}{q}}, V^{-\frac{p'}{q}} \in \T{A}_{q', p'}$, we have that $U^{-\frac{p'}{q}}, V^{-\frac{p'}{q}} \in \T{A}_{p'}$. The fact that $V \in  \T{A}{}_{p, q}$ tells us that (by the A${}_{p, q}$ property and \eqref{MatrixHolder})

\begin{align*} (1) & = \sup_Q \fint_Q \norm{\MC{U}_Q ^{-1} (B^*(x) - m_Q B^*) \MC{V}_Q} \, dx  \approx \sup_Q \fint_Q \norm{\MC{U}_Q ' (B^*(x) - m_Q B^*) (\MC{V}_Q ')^{-1}} \, dx \end{align*}

which by \eqref{DualRedEq1} is nothing but $(1)$  with respect to matrix weights $ V^{-\frac{p'}{q}}, U^{-\frac{p'}{q}},$ the symbol $ B^*$, and the exponent $p'$,  and thus $(1)$  equivalent to $(2)$ with respect to $ V^{-\frac{p'}{q}}, U^{-\frac{p'}{q}}, B^*,$ and $ p'$, which (again by \eqref{DualRedEq1}) is nothing but $(3)$.  This tells us that $(3) \Longleftrightarrow (1) \Longleftrightarrow (2) \Longleftrightarrow(6).$

Furthermore, $(6) \leq (5)$ by an easy application of H\"{o}lder's inequality, since $p'/q' > 1$.  Thus, if we can show that $(5) \lesssim (2) + (3)$ then we will have $(3) \Longleftrightarrow (1) \Longleftrightarrow (2) \Longleftrightarrow(6) \Longleftrightarrow (5)$.  To that end,

\begin{align*} & \pr{\fint_Q \pr{ \fint_Q \norm{V^\frac{1}{q} (x)  (B(x) - B(y)) U^{-\frac{1}{q}} (y) } ^{q}  \, dx}^\frac{p'}{q} \, dy}^\frac{1}{p'}  \\ & \phantom{WWW} \lesssim \pr{ \fint_Q \pr{ \fint_Q \norm{V^\frac{1}{q} (x)  (B(x) - m_Q B) U^{-\frac{1}{q}} (y) } ^{q}  \, dx}^\frac{p'}{q} \, dy}^\frac{1}{p'}
\\ & \phantom{WWWWW} + \pr{ \fint_Q \pr{ \fint_Q \norm{V^\frac{1}{q} (x)  (B(y) -m_Q  B) U^{-\frac{1}{q}} (y) } ^{q}  \, dx}^\frac{p'}{q} \, dy}^\frac{1}{p'} =   (A) + (B).
\end{align*}
  Using the  matrix A${}_{p, q}$ property we get
\begin{align*} (A) & \leq  \pr{ \fint_Q \pr{ \fint_Q \norm{V^\frac{1}{q} (x)  (B(x) - m_Q B) \MC{U}_Q ^{-1} }^q  \norm{ \MC{U}_Q U^{-\frac{1}{q}} (y) } ^{q}  \, dx}^\frac{p'}{q} \, dy}^\frac{1}{p'}
\\ & = \pr{ \fint_Q \pr{ \fint_Q \norm{V^\frac{1}{q} (x)  (B(x) - m_Q B) \MC{U}_Q ^{-1} }^q    \, dx}^\frac{p'}{q} \norm{ \MC{U}_Q U^{-\frac{1}{q}} (y) }^{p'} \, dy}^\frac{1}{p'} \lesssim \norm{U}_{\text{A}_{p, q}} ^\frac{1}{p'} (2). \end{align*}
and likewise
\begin{align*}
(B) & \leq \pr{ \fint_Q \pr{ \fint_Q \norm{V^\frac{1}{q} (x) \MC{V}_Q '}^q  \norm{(\MC{V}_Q ')^{-1}  (B(y) -m_Q  B) U^{-\frac{1}{q}} (y) } ^{q}  \, dx}^\frac{p'}{q} \, dy}^\frac{1}{p'}
\\ & = \pr{ \fint_Q \pr{ \fint_Q \norm{V^\frac{1}{q} (x) \MC{V}_Q '}^q   \, dx}^\frac{p'}{q} \norm{(\MC{V}_Q ')^{-1}  (B(y) -m_Q  B) U^{-\frac{1}{q}} (y) } ^{p'}  \, dy}^\frac{1}{p'} \lesssim \Apq{V}^\frac{1}{p'} (3) \end{align*} Finally, as was observed already, $(1)$ is equivalent to $(1)$  with respect to $ V^{-\frac{p'}{q}}, U^{-\frac{p'}{q}},  B^*,$ and $p'$, which is equivalent to $(5)$ with respect $ V^{-\frac{p'}{q}}, U^{-\frac{p'}{q}},  B^*,$ and $p'$, which is easily be seen to be nothing but $(4)$.
\end{proof}

\section{Commutator lower bound: proof of Theorem \ref{FracLowerBoundThm}}  \label{WienerSection}
In this section we will prove Theorem \ref{FracLowerBoundThm} and in the process prove that $I_\alpha$ is a fractional lower bound operator (which will complete the proof of Theorem \ref{BloomFrac}).  As stated in the introduction, this will be done by modifying Wiener algebra arguments from by \cite{IPT,LT}. Let $W$ be a matrix weight and suppose that $\vec{f} \in L^p \cap L^p\left( W^{\frac{p}{q}}\right)$ and $\vec{g} \in L^q \cap L^{q'}\left( W^{-\frac{q'}{q}}\right)$. Let $E \subset \reals^d$ be measurable. For any $t \in \reals^d$, define \[k_{\alpha,t}(x,y) := e^{-2\pi it \cdot x} k_\alpha(x,y) e^{2\pi it \cdot y} \] where $k_\al(x, y) = |x-y|^{\al - d}$.  We then have from H\"{o}lder's inequality  that
\begin{align*}
    &\myabs{\int_{\reals^d} \int_{\reals^d} \chi_{E \times E}(x,y) k_{\alpha,t}(x,y) \myinnerprod{\vec{f}(y)}{\vec{g}(x)}_{\cplx^n} \dy \dx} \nonumber \\
    &\phantom{W}= \myabs{\int_{\reals^d} \int_{\reals^d} \chi_{E \times E}(x,y) k_{\alpha,t}(x,y) \myinnerprod{W^{\frac{1}{q}} (x) \vec{f}(y)}{W^{-\frac{1}{q}} (x) \vec{g}(x)}_{\cplx^n} \dy \dx} \\
    &\phantom{W}\leq \mynorm{\chi_E I_\alpha \chi_E \, }_{L^p\left( W^\frac{p}{q} \right) \longrightarrow L^q(W)} ||\vec{f} \, ||_{L^p\left( W^\frac{p}{q} \right)} \mynorm{\vec{g} \, }_{L^{q'}\left( W^{-\frac{q'}{q}}\right)}
\end{align*}

\noindent Thus, if $\psi = \hat{\rho}$ for $\rho \in L^1(\Rd)$ (where $\hat{}$ denotes Fourier transform) then,
\begin{align}
    &\myabs{\int_{\reals^d} \int_{\reals^d} \psi\left( \frac{x-y}{\epsilon} \right) \chi_{E \times E}(x,y) k_{\alpha}(x,y) \myinnerprod{\vec{f}(y)}{\vec{g}(x)}_{\cplx^n} \dy \dx}  \nonumber  \\
    &\phantom{W}=  \myabs{\int_{\reals^d} \int_{\reals^d} \epsilon^d \left[ \int_{\reals^d} \rho(\epsilon t)e^{-2\pi it\cdot(x-y)} \dt \right]\chi_{E \times E}(x,y) k_{\alpha}(x,y) \myinnerprod{\vec{f}(y)}{\vec{g}(x)}_{\cplx^n} \dy \dx}  \nonumber \\
    &\phantom{W}=\mynorm{\chi_E I_\alpha \chi_E \, }_{L^p\left( W^\frac{p}{q} \right) \longrightarrow L^q(W)} ||\vec{f} \, ||_{L^p\left( W^\frac{p}{q} \right)} \mynorm{\vec{g} \, }_{L^{q'}\left( W^{-\frac{q'}{q}}\right)} \mynorm{\rho}_{L^1(\reals^d)}, \label{WeinerIneq}
\end{align}

\noindent
which means that \eqref{WeinerIneq} holds if $\psi$ is in the Wiener algebra $ W_0(\reals^d) := \{\hat{\rho} : \rho \in L^1(\Rd)\}$. To prove Theorem \ref{FracLowerBoundThm} we will use \eqref{WeinerIneq} with $\psi(x)  =  |x|^{d-\alpha} \phi(x)$ where $\phi \in C_c^\infty(\Rd)$.  While it is likely known that such a function lies in $W_0(\Rd)$, a precise reference seems difficult to find and thus we will prove it (using a simple idea from \cite{DT}) for the sake of completeness.

\begin{prop} \label{prop:wein}
Let $0 < \alpha < d$. If $\phi \in C_c^\infty(\Rd)$, then $\myabs{\cdot}^{d-\alpha} \phi \in W_0(\reals^d)$.
\end{prop}

\begin{proof}
Let $F(x) = \myabs{x}^{d-\alpha}$ and ${\MS{F}}(x) = \myabs{x}^{d-\alpha} \phi(x)$.   If $\be \in \myset{0,1}^d$,  then by an easy induction we have  $$D^\be F (x) = {\Phi_\be (x)}{\myabs{x}^{{d-\al} - 2\myabs{\be }}}$$ where $\Phi_\be(x)$ is the sum of monomials of degree $\myabs{\be}$ in $d$ variables, which means that \begin{equation}
\myabs{D^\be {\MS{F}} (x) } \lesssim \myabs{x}^{d-\alpha - \myabs{\be}} \label{DerivEst} \end{equation}

Now let $1 < \delta < \min \myset{1 + \frac{d - \alpha}{\alpha},2}$.  Using H\"{o}lder's inequality, a standard integration by parts argument, and the Hausdorff-Young inequality, we then have by an argument identical to the one used in \cite[Lemma $3.2$]{IPT} that
\begin{equation*}
    \mynorm{\reallywidehat{{\MS{F}}}}_{L^1(\reals^d)} \lesssim   \left( \int_{\Rd} \myabs{D^\be {\MS{F}}(x)}^{\delta} \dx \right)^{\frac{1}{\delta}}  < \infty
\end{equation*}
which is finite by \eqref{DerivEst}, since \begin{equation*}
    \delta[|\be|-(d-\alpha)] < \left( 1 + \frac{d-\alpha}{\alpha}\right)[d-(d-\alpha)] = \left( 1 + \frac{d-\alpha}{\alpha}\right)\alpha = d.
\end{equation*}
Fourier inversion now immediately completes the proof.  \end{proof}

\noindent
Applying  \eqref{WeinerIneq} with $\psi(x)  =  |x|^{d-\alpha} \phi(x)$ where $\phi \in C_c^\infty (\Rd)$, we obtain the inequality
\begin{align}
    &\myabs{\int_E \int_E \frac{1}{\epsilon^{d-\alpha}} \, \phi\left( \frac{x-y}{\epsilon} \right) \myinnerprod{\vec{f}(y)}{\vec{g}(x)}_{\cplx^n} \dy \dx } \nonumber \\
    &\phantom{WWW}\leq     \mynorm{\chi_E I_\alpha \chi_E \, }_{L^p\left( W^\frac{p}{q} \right) \longrightarrow L^q(W)} ||\vec{f} \, ||_{L^p\left( W^\frac{p}{q} \right)} \mynorm{\vec{g} \, }_{L^{q'}\left( W^{-\frac{q'}{q}}\right)} \mynorm{\rho}_{L^1(\reals^d)}. \label{eq-res-4}
\end{align}
\noindent
We proceed by citing a uniform boundedness result which was (implicitly) proved in \cite[Proposition 3.1]{IM}.

\begin{prop} \label{thm-frac-averaging}
Let $E$ be measurable with $0 < |E| < \infty$ and define the fractional averaging operator by $\mymor{A_E}{L^p\left( W^{\frac{p}{q}}\right)}{L^q(W)}$ by \[A_E \vec{f} := \frac{\chi_E}{|E|^{1-\frac{\alpha}{d}}} \int_E \vec{f}(x) \dx, \]
then
\begin{equation} \label{eq-technicality1}
    \mynorm{ \MC{W}_E ' \MC{W}_E } \approx \mynorm{A_E}_{L^p\left( W^{\frac{p}{q}}\right) \longrightarrow L^q(W)}.
\end{equation}
%
\end{prop}
\noindent
We can now state and prove the main technical lemma of this section, which immediately proves that $I_\al$ is a fractional lower bound operator.
\begin{lem} \label{thm-main-lower-bound}
If $B$ is a ball and $E \subset B$ with $|E| > 0$, then
\[ \mynorm{ \MC{W}_E ' \MC{W}_E } \lesssim  \left[ \frac{|B|}{|E|} \right]^{1 - \frac{\alpha}{d}} \mynorm{\chi_E I_\alpha \chi_E}_{L^p\left( W^{\frac{p}{q}}\right) \longrightarrow L^q(W)}. \]
\end{lem}

\begin{proof}
Suppose that $\vec{f} \in L^p \cap L^p\left( W^{\frac{p}{q}}\right)$ and $\vec{g} \in L^{q'} \cap L^{q'}\left( W^{-\frac{q'}{q}}\right)$ Clearly we have
\begin{align*}
  \left[ \frac{|E|}{|B|} \right]^{1 - \frac{\alpha}{d}} \myinnerprod{A_E\vec{f}}{\vec{g}}_{L^2(\reals^d)} =   \frac{1}{|B|^{1 - \frac{\alpha}{d}}} \int_{E} \int_{E} \myinnerprod{\vec{f}(y)}{\vec{g}(x)}_{\cplx^n} \dy \dx.
\end{align*}
\noindent
Let $B$ have radius $\epsilon$ and pick $\phi \in C_c ^\infty(\Rd)$ such that $\phi =  1$ on the open ball $B(0,2)$. If $x,y \in B$ then $\myabs{x-y} \leq  2\epsilon$ and therefore \eqref{eq-res-4} gives us that
\begin{align}
    & \myabs{\frac{1}{|B|^{1 - \frac{\alpha}{d}}} \int_{E} \int_{E} \myinnerprod{\vec{f}(y)}{\vec{g}(x)}_{\cplx^n} \dy \dx}  \approx \myabs{\frac{1}{(\epsilon^d)^{1 - \frac{\alpha}{d}}} \int_{E} \int_{E} \myinnerprod{\vec{f}(y)}{\vec{g}(x)}_{\cplx^n} \dy \dx} \nonumber \\
    &\phantom{WWW}\lesssim \mynorm{\chi_E I_\alpha \chi_E \, }_{L^p\left( W^\frac{p}{q} \right) \longrightarrow L^q(W)} ||  \vec{f} \, ||_{L^p\left( W^\frac{p}{q} \right)} \mynorm{ \vec{g} \, }_{L^{q'}\left( W^{-\frac{q'}{q}}\right)}  \nonumber
\end{align}
which means
\begin{align}
    \myabs{ \myinnerprod{A_E\vec{f}}{\vec{g}}_{L^2(\reals^d)} } \lesssim \left[ \frac{|B|}{|E|} \right]^{1-\frac{\alpha}{d}} \mynorm{\chi_E I_\alpha \chi_E \, }_{L^p\left( W^\frac{p}{q} \right) \longrightarrow L^q(W)} || \vec{f} \, ||_{L^p\left( W^\frac{p}{q} \right)} \mynorm{\vec{g} \, }_{L^{q'}\left( W^{-\frac{q'}{q}}\right)}.
\end{align}

Duality,  the density of $L^p \cap L^p\left( W^{\frac{p}{q}}\right)$ in $L^p\left( W^{\frac{p}{q}}\right)$ and $ L^q \cap L^{q'}\left( W^{-\frac{q'}{q}}\right)$ in $ L^{q'}\left( W^{-\frac{q'}{q}}\right)$ (see \cite[Proposition 3.7]{CMR}), and \eqref{eq-technicality1} now completes the proof.
\end{proof}
\begin{proof} [Proof of Theorem \ref{FracLowerBoundThm}]
We assume that $\mynorm{\left[M_B,I_\alpha \right]}_{L^p\left( U^{\frac{p}{q}}\right) \longrightarrow L^q(V)} < \infty$, or the lower bound holds trivially. Let $B$ be a ball and, for each $M>0$, define
\begin{equation}
    E_M := \left\{ x \in B : \max \left\{ \mynorm{U(x)}, \mynorm{U^{-1}(x)}, \mynorm{V(x)}, \mynorm{V^{-1}(x)} \right\} < M \right\}.
\end{equation}
By continuity of Lebesgue measure, there exists $M$ such that $2|E_M| > |B|$. Further, define
\begin{equation}
    \mynorm{W}_{A_{p,q}(E_M)} := \fint_{E_M} \left( \, \fint_{E_M} \mynorm{W^{\frac{1}{q}}(x)W^{-\frac{1}{q}}(y)}^{p'} \dy \right)^{\frac{q}{p'}} \dx
\end{equation}
and
\begin{equation}
    \mynorm{B}_{\widetilde{BMO}^{p,q}_{V,U}({E_M})} := \left( \fint_{E_M} \left( \, \fint_{E_M} \mynorm{V^{\frac{1}{q}}(x) (B(x)-B(y)) U^{-\frac{1}{q}}(y)}^{p'} \dy \right)^{\frac{q}{p'}} \dx \right)^{\frac{1}{q}}.
\end{equation}
Let $W$ and $\Phi$ be defined as in Subsection \ref{IntUBFracSec}.  Using ideas from that subsection, we have
\begin{align}
    \mynorm{W}_{A_{p,q}({E_M})}
        &= \fint_{E_M} \left( \, \fint_{E_M} \mynorm{W^{\frac{1}{q}}(x)W^{-\frac{1}{q}}(y)}^{p'} \dy \right)^{\frac{q}{p'}} \dx = \fint_{E_M} \left( \, \fint_{E_M} \mynorm{\Phi(x)\Phi^{-1}(y)}^{p'} \dy \right)^{\frac{q}{p'}} \dx \nonumber \\
        &\approx \mynorm{U}_{A_{p,q}({E_M})} + \mynorm{V}_{A_{p,q}({E_M})} + \mynorm{B}^q_{\widetilde{BMO}^{p,q}_{V,U}({E_M})}, \nonumber
\end{align}
so that
\begin{align}
    &\left( \mynorm{U}_{A_{p,q}(E_M)} + \mynorm{V}_{A_{p,q}(E_M)} + \mynorm{B}^q_{\widetilde{BMO}^{p,q}_{V,U}(E_M)}  \right)^\frac{1}{q}  \approx \mynorm{W}^{\frac{1}{q}}_{A_{p,q}(E_M)} \approx \mynorm{\MC{W}_{E_M} ' \MC{W}_{E_M}} \nonumber \\
    &\phantom{WWW}\lesssim \left[ \frac{|B|}{|E_M|} \right]^{1-\frac{\alpha}{d}} \mynorm{\chi_{E_M}I_\alpha \chi_{E_M}}_{L^p\left( W^{\frac{p}{q}}\right) \longrightarrow L^q(W)} \nonumber \\
    &\phantom{WWW}\lesssim \mynorm{\left[M_B,I_\alpha \right]}_{L^p\left( U^{\frac{p}{q}}\right) \longrightarrow L^q(V)} \nonumber \\
        &\phantom{WWWWW}+ \mynorm{\chi_{E_M}I_\alpha \chi_{E_M}}_{L^p\left( U^{\frac{p}{q}}\right) \longrightarrow L^q(U)}
        + \mynorm{\chi_{E_M}I_\alpha \chi_{E_M}}_{L^p\left( V^{\frac{p}{q}}\right) \longrightarrow L^q(V)}. \nonumber
\end{align}
By assumption, all quantities above are finite so we can rescale with the replacement $B \mapsto rB$, for $r>0$. Upon dividing by $r$ and taking $r \longrightarrow \infty$, we obtain
\begin{equation*}
    \mynorm{B}_{\widetilde{BMO}^{p,q}_{V,U}(E_M)} \lesssim \mynorm{\left[M_B,I_\alpha \right]}_{L^p\left( U^{\frac{p}{q}}\right) \longrightarrow L^q(V)}.
\end{equation*}
Applying Fatou's lemma with $M \longrightarrow \infty$ and taking the supremum over all balls $B$ in $\reals^d$, we obtain
\begin{align}
   \mynorm{B}_{\widetilde{BMO}^{p,q}_{V,U}}
        \lesssim \mynorm{\left[M_B,I_\alpha \right]}_{L^p\left( U^{\frac{p}{q}}\right) \longrightarrow L^q(V)}, \nonumber
\end{align}
and a slight modification to the arguments above proves that $$\|B\|_{\BMOVUTpqd} \lesssim \mynorm{\left[M_B,I_\alpha \right]}_{L^p\left( U^{\frac{p}{q}}\right) \longrightarrow L^q(V)}$$ which proves Theorem \ref{FracLowerBoundThm}. \end{proof}
\section{Proof of theorem \ref{FracOrliczProp}} \label{OrliczUpBound}
Our proof is a combination and modification of the arguments in  \cite{IPT,CIM,LI}. For additional information on Orlicz spaces, see e.g., \cite{BerLof}.


\begin{prop} There exists $2^d$ dyadic grids $\D_t, t \in \{0, 1\}^d$ where
\begin{align*}
    &\innp{V ^\frac{1}{q} [M_B, I_\al] U^{-\frac{1}{q}} \V{f}, {\V{g}}}_{L^2}  \\
    &\phantom{WW}  \lesssim  \sum_{t \in \{0, \frac13\}^d} \sum_{Q \in \D^t} \frac{1}{|Q| ^{1 - \frac{\alpha}{d}}} \int_Q \int_Q \left|\innp{ V ^\frac{1}{q} (x)  (B(x) - B(y)) U^{-\frac{1}{q}}(y)\V{f}(y),  \V{g}(x) } _{\cplx^n} \right| \, dx \, dy.
\end{align*}
\end{prop}

\begin{proof} Noting that $V ^{\frac{1}{q}} [M_B, I_\al] U^{-\frac{1}{q}}$ is an integral operator with kernel $V ^\frac{1}{q}  (x) (B(x) - B(y)) U^{-\frac{1}{q}}(y) |x-y|^{\al - d}$, the proof is almost identical to the proof of \cite[Lemma $3.8$]{IM}.   \end{proof}

As in \cite{LI}, we will make use of the following well known fact about Orlicz spaces \begin{prop} \label{OrliczProp}  For an increasing, convex function $\Phi$, if $$ \norm{f}_{\Phi, Q} = \inf \set{\lambda > 0 : \fint_Q \Phi\pr{\frac{|f(y)|}{\lambda}} \, dy \leq 1} \text{ and } \norm{f}_{\Phi, Q} ^* = \inf_{s > 0} \set{ s + s \fint_Q \Phi\pr{\frac{|{f}(y)|}{s}} \, dy
}$$ then $\norm{{f}}_{\Phi, Q}  \leq \norm{{f}}_{\Phi, Q} ^*  \leq 2 \norm{{f}}_{\Phi, Q} $
\end{prop}

\begin{proof}[Proof of Proposition \ref{FracOrliczProp}] Clearly it is enough to prove for a dyadic grid $\D$ that
\begin{align*}
     \sum_{Q \in \D} {|Q| ^{
      \frac{\alpha}{d}}} \fint_Q \int_Q \left|\innp{ V ^{\frac{1}{q}} (x)
      (B(x) - B(y)) U ^{-\frac{1}{q}} (y) \V{f}(y),  \V{g}(x) } _{\cplx^n} \right| \, dx \, dy  \lesssim \min\{\kappa_1, \kappa_2\} \|\V{f}\|_{L^p} \|\V{g}\|_{L^{q'}}.
\end{align*}
Also, by Fatou's lemma, it is enough to assume that $\V{f}, \V{g}$ are bounded with compact support. For that matter, the generalized H\"{o}lder inequality gives

\begin{align*}
    & \sum_{Q \in \D} {|Q| ^{
      \frac{\alpha}{d}}} \fint_Q \int_Q \left|\innp{ V ^{\frac{1}{q}} (x)
      (B(x) - B(y)) U ^{-\frac{1}{p}} (y) \V{f}(y),  \V{g}(x) } _{\cplx^n} \right| \, dx \, dy \\
    &\phantom{WWWWWW}  \leq \min\{\kappa_1, \kappa_2\} \sum_{Q \in \D} {|Q| ^{1+ \frac{\alpha}{d}}} \|\V{f}\|_{\ol{D}, Q} \|\V{g}\|_{\ol{C}, Q}
\end{align*}
Fix $a>2^{d+1}$ and define the collection of cubes
 \begin{equation*}
\MC{Q} ^k =
  \{Q \in \D : a^k < \|\V{f}\|_{\bar{D}, Q} \leq a^{k + 1}\},
\end{equation*}
and let $\MC{S}^k$ be the disjoint collection of
$Q\in \D$ that are maximal with respect to the
inequality $\|{f}\|_{\bar{D}, Q} > a^k$. Note that since $\V{f}$ is bounded with compact support, such maximal cubes are guaranteed to exist.
Set $\MC{S} = \bigcup_k \MC{S}^k$.
We now continue the above estimate:
\begin{align}
&  \sum_k \sum_{Q \in \MC{Q}^k} |Q| ^{ 1+ \frac{\alpha}{d}}\|\V{f}\|_{\ol{D}, Q} \|\V{g}\|_{\ol{C}, Q}
\leq  \sum_k  a^{k + 1} \sum_{Q \in \MC{Q}^k} |Q| ^{1 +
  \frac{\alpha}{d}}  \|\V{g}\|_{\ol{C}, Q} \nonumber  \\
& \qquad \qquad = \sum_k  a^{k + 1} \sum_{P \in \MC{S}^k} \sum_{\substack{Q \in
  \MC{Q}^k \\ Q \subset P}} |Q| ^{ 1+ \frac{\alpha}{d}} \|\V{g}\|_{\ol{C}, Q} \label{OrUpBoundEst1}.
\end{align}

We now estimate the inner sum.  Let $\ell(P) = 2^{-m_0}$, so

\begin{align}  \sum_{Q \subset P} & |Q| ^{1+ \frac{\alpha}{d}} \|\V{g}\|_{\ol{C}, Q}   = \sum_{m = m_0} ^\infty 2^{-m \al} \sum_{\substack{\ell(Q) = 2^{-m} \\ Q \subset P}} |Q|\|\V{g}\|_{\ol{C}, Q} \label{OrUpBoundEst2}
\end{align}

However, Proposition \ref{OrliczProp} gives us that
\begin{align*} \sum_{\substack{\ell(Q) = 2^{-m} \\ Q \subset P}} |Q|\|g\|_{\ol{C}, Q}  & \leq  \sum_{\substack{\ell(Q) = 2^{-m} \\ Q \subset P}} |Q| \inf_{s>0} \set{ s + s \fint_Q \ol{C}\pr{\frac{|\V{f}(y)|}{s}} \, dy }
\\ & \leq \inf_{s > 0}\set{ \sum_{\substack{\ell(Q) = 2^{-m} \\ Q \subset P}} \left[|Q|   s + s \int_Q \ol{C}\pr{\frac{|\V{f}(y)|}{s}} \, dy \right]}
\\ & = \inf_{s > 0}\set{ |P|   s + s \int_P \ol{C}\pr{\frac{|\V{f}(y)|}{s}} \, dy }
\\ &  = |P| \inf_{s > 0} \set{  s + s \fint_P \ol{C}\pr{\frac{|\V{f}(y)|}{s}} \, dy }
\leq 2 |P|\|g\|_{\ol{C}, P}
\end{align*}

Plugging this into \eqref{OrUpBoundEst2} we get

\begin{align*} \sum_{m = m_0} ^\infty & 2^{-m \al} \sum_{\substack{\ell(Q) = 2^{-m} \\ Q \subset P}} |Q|\|\V{g}\|_{\ol{C}, Q}
\leq 2 |P|\|\V{g}\|_{\ol{C}, P} \sum_{m = m_0} ^\infty  2^{-m \al}
 \lesssim |P|^{1+\frac{\al}{d}} \|\V{g}\|_{\ol{C}, P}
\end{align*}

and combining this with \eqref{OrUpBoundEst1} gives

\begin{align*} \sum_k   a^{k + 1}  &\sum_{P \in \MC{S}^k} \sum_{\substack{Q \in
  \MC{Q}^k \\ Q \subset P}} |Q| ^{ 1+ \frac{\alpha}{d}} \|\V{g}\|_{\ol{C}, Q}
  \lesssim  \sum_k  \sum_{P \in \MC{S}^k}     |P|^{1+\frac{\al}{d}} \|\V{f}\|_{\ol{D}, P} \|\v{g}\|_{\ol{C}, P}
  \\ & = \sum_{P \in \MC{S}} |P| \pr{|P|^{\frac{\al}{d}}\|f\|_{\ol{D}, P}} \pr{ \|\V{g}\|_{\ol{C}, P}}
  \leq \sum_{P \in \MC{S}} |P| \inf _{ x \in P}    M_{\ol{D}} ^\al \V{f} (x)  M_{\ol{C}} \V{g} (x) \end{align*}
  \noindent where as usual,  $M_{\ol{D}} ^\al \V{f} (x) = \sup_{\D \ni P \ni x} |P|^{\frac{\al}{d}}\|f\|_{\ol{D}, P} $ and when $\al = 0$ we omit the superscript.

 For each  $P \in \MC{S}$, define
 \begin{equation*}
E_P = P \backslash \bigcup_{\substack{P' \in \MC{S} \\ P' \subsetneq
    P}} P'.
\end{equation*}
Then by Proposition $A.1$ in \cite{CruzWeights}, since $\V{f}, \V{g}$ are bounded with compact support, we have that the sets $E_P$ are pairwise
disjoint and $|E_P|\geq \frac{1}{2}|P|$.   Thus,
\begin{align*}
\sum_{P \in \MC{S}} & |P| \inf _{ x \in P}    M_{\ol{D}} ^\al \V{f} (x)  M_{\ol{C}} \V{g} (x)
 \leq 2   \sum_{P \in \MC{S}} |E_P| \inf _{ x \in P}    M_{\ol{D}} ^\al \V{f} (x)  M_{\ol{C}} \V{g} (x)      \\
 & \leq 2   \sum_{P \in \MC{S}} \int_{E_P}    M_{\ol{D}} ^\al \V{f} (x)  M_{\ol{C}} \V{g} (x)  \, dx
     \leq 2 \int_{\mathbb{R}^d} M^\al _{{\ol{D}}}  \V{f} (x)  \, M_{\ol{C}} \V{g} (x) \, dx
    \\ & \leq 2  \|M_{{\ol{D}}} ^\al \V{f}\|_{L^q} \|M_{{\ol{C}}} \V{g}\|_{L^{q'}}
      \lesssim \|\V{f}\|_{L^{p}} \|\V{g}\|_{L^{q'}}.
\end{align*}

\end{proof}

\bibliographystyle{amsalpha}
\bibliography{MatrixWeights}

 \end{document}